%% file: CharPropForFDTC.tex
\documentclass[%11pt,%BCOR=1cm,
a4paper
]{amsart}
\usepackage{mathtools}
\usepackage{amssymb, amsthm, amsmath, pifont} %f{\"u}r Symbole der reellen Zahlen
\usepackage[usenames]{color}
\usepackage{paralist}
\usepackage{amsfonts}
\usepackage{enumerate}
\usepackage{pdfpages}
\usepackage[all]{xypic}
\usepackage{graphicx}
\usepackage{psfrag}
\usepackage{verbatim}
\usepackage{pstricks}
\usepackage[latin1]{inputenc}
\usepackage{mathrsfs}
\usepackage[all,knot]{xy}
\newtheorem{thm}{Theorem}%[section]
\newtheorem{corollary}[thm]{Corollary}

\newtheorem{lemma}[thm]{Lemma}
\newtheorem{question}[thm]{Question}
\newtheorem{example}[thm]{Example}
\newtheorem{definition}[thm]{Definition}
\newtheorem*{intdef}{Definition}
\newtheorem*{claim*}{Claim}
\newtheorem{remark}[thm]{Remark}
\newtheorem{prop}[thm]{Proposition}
\newtheorem{claim}[thm]{Claim}
\newtheorem{obs}[thm]{Observation}

\newenvironment{Example}{\begin{example}\rm}{\end{example}}

\newenvironment{Remark}{\begin{remark}\rm}{\end{remark}}

\newenvironment{Question}{\begin{question}\rm}{\end{question}}
\def\et{\;\mbox{and}\;}

\def\wr{{\rm{wr}}}

\def\fd{{\omega}}
\def\d{{\omega_*}}

\def\et{\quad\mbox{and}\quad}

\def\epsilon{\varepsilon}

\def\R{\mathbb{R}}
\def\Z{\mathbb{Z}}
\def\C{\mathbb{C}}

\begin{document}
\title[The slice-Bennequin inequality for the FDTC]
{The slice-Bennequin inequality for the fractional Dehn twist coefficient}\author{Peter Feller}

\address{ETH Z\"urich, R\"amistrasse 101, 8092 Z\"urich, Switzerland}
\email{peter.feller@math.ch}

\begin{abstract}We characterize the fractional Dehn twist coefficient (FDTC) on the $n$-stranded braid group as the unique homogeneous quasimorphism to $\R$ of defect at most 1 that equals 1 on the positive full twist and vanishes on the $(n-1)$-stranded braid subgroup.
In a different direction, we establish that the slice-Bennequin inequality holds with the FDTC in place of the writhe. In other words, we establish an affine linear lower bound for the smooth slice genus of the closure of a braid in terms of the braid's FDTC. We also discuss connections between these two seemingly unrelated results.
In the appendix we provide a unifying framework for the slice-Bennequin inequality and its counterpart for the FDTC.
\end{abstract}
\maketitle

\section{Introduction}

A \emph{quasimorphism} on a group $G$ is a function $f$ from $G$ to the real numbers $\R$ such that
$\sup_{a,b\in G}|f(ab)-f(a)-f(b)|<\infty$,
where $\sup_{a,b\in G}|f(ab)-f(a)-f(b)|$ is called the \emph{defect} of $f$ and is denoted by $D_f$.
A function $f\colon G\to\R$ is said to be \emph{homogeneous}
if $f(g^k)=kf(g)$ for all $g\in G$ and integers $k$.
In this article we focus on the \emph{fractional Dehn twist coefficient (FDTC)}, a certain homogeneous quasimorphism on the braid group
on $n$ strands.
The FDTC appears in several contexts concerning different aspects of low-dimensional topology; see for example Gabai-Oertel, Malyutin, and Honda-Kazez-Mati\'c
%, in several different contexts concerning 3-manifolds and links in them
\cite{Gabai_EssentialLaminations3Manifolds,
Malyutin_Twistnumber,
HondaKazezMatic_RightVeeringI,
HondaKazezMatic_RightVeeringII}.

\subsection*{A characterization of the FDTC as the homogeneous quasimorphism of smallest defect}
For a fixed integer $n\geq 1$, we denote by
\[B_n=\left\langle a_1,\cdots,a_{n-1}\;\middle|\;a_ia_j=a_ja_i\text{ for }|i-j|\geq 2,a_ia_{i+1}a_i=a_{i+1}a_ia_{i+1}\right\rangle,\] \emph{Artin's braid group}~\cite{Artin_TheorieDerZoepfe}. For the entire text, we identify $B_{n-1}\subset B_n$ as a subgroup via the inclusion $\iota\colon B_{n-1}\to B_n,a_i\mapsto a_{i+1}$, whenever $n\geq 2$.
We delay an explicit definition of the FDTC. However, we recall that
the FDTC, denoted by $\fd\colon B_n\to \R$, is known to be a homogeneous quasimorphism of defect $\leq 1$ (in fact it is know to have defect 1 when $n\geq3$; compare Lemma~\ref{lemma:qmhaveD>=1}) that satisfies $\fd(B_{n-1})=\{0\}$ and $\fd\left(\Delta^2\right)=1$; see~\cite{Malyutin_Twistnumber}. Here $\Delta^2$ denotes $(a_1a_2\cdots a_{n-1})^n\in B_n$, which is known as the \emph{positive full twist} and, for $n\geq 3$, generates the center of $B_n$. We establish that these properties  characterize the FDTC.

\begin{thm}\label{thm:CharPropForFDTC} For every integer $n\geq 3$,
there exists a unique homogeneous quasimorphism
$\fd\colon B_n\to \R$ with defect at most 1
that satisfies the following properties:
\begin{inparaenum}[(i)]
  \item\label{item:norm} %(normalization)
  $\fd\left(\Delta^2\right)=1$ and 
  \item\label{item:Bn-1=0} $\fd(\beta)=0$ for all $\beta \in B_{n-1}\subset B_n$.
  \end{inparaenum}
\end{thm}

We exclude considerations for $n=2$  (and $n=1$) as it is clear that there is at most one homogeneous quasimorphism on the infinite cyclic group $B_2$ (and the trivial group $B_1$) that sends $\Delta^2$ to a given value since homogeneous quasimorphisms on Abelian groups are group homomorphisms. 
In contrast to this and to Theorem~\ref{thm:CharPropForFDTC}, for $n\geq 3$ there are many homogeneous quasimorphisms on $B_n$.

\begin{prop}\label{prop:uncqms}For every integer $n\geq 3$ and every $\epsilon>0$,
there exist continuum-many linearly independent homogeneous quasimorphisms $f\colon B_n\to\R$ with defect at most $1+\epsilon$ that satisfy (i) $f\left(\Delta^2\right)=1$ and (ii) $f(B_n)=\{0\}$.
\end{prop}

We briefly comment on the two assumptions~\eqref{item:norm} and~\eqref{item:Bn-1=0}.

\eqref{item:norm} can be understood as a normalization condition.
In other words, Theorem~\ref{thm:CharPropForFDTC} says that the homogeneous quasimorphisms $f\colon B_n\to \R$ that satisfy \[D_{f}\leq \left|f(\Delta^2)\right|\et f(B_{n-1})=\{0\}\] form a 1-dimensional $\R$-subspace of the vector space of functions from $B_n$ to $\R$,
while Proposition~\ref{prop:uncqms} says that the $\R$-subspace generated by homogeneous quasimorphisms $f\colon B_n\to \R$ that satisfy \[D_{f}\leq (1+\epsilon)\left|f(\Delta^2)\right|\et f(B_{n-1})=\{0\}\] has uncountably infinite dimension. %As an aside we note that if a homogeneous quasimorphisms $\d\colon B_n\to \R$ satisfies $\d(B_{n-1})=\{0\}$, then $D_{\d}\geq \left|\d(\Delta^2)\right|$; see Lemma~\ref{lemma:qmhaveD>=1}.

%Homogeneous quasimorphisms that satisfy~\eqref{item:Bn-1=0} were studied by Malyutin under the name \emph{kernel pseudocharacters}~\cite{Malyutin_09_pchar}.
Every homogeneous quasimorphism $f\colon B_n\to \R$ can be written as the sum of a homogeneous quasimorphism that satisfy~\eqref{item:Bn-1=0} and a homogeneous quasimorphism that is determined by the homogeneous quasimorphism ${B_{n-1}}\to \R, \beta\mapsto f(\iota(\beta))$~\cite[Theorem~2]{Malyutin_09_pchar}. So, informally speaking, understanding homogeneous quasimorphisms on $B_n$ amounts to understanding homogeneous quasimorphisms that satisfy~\eqref{item:Bn-1=0} on $B_n$ and homogeneous quasimorphisms on $B_{n-1}$.

\subsection*{The slice-Bennequin inequality for the FDTC}
For a link $L$---a non-empty oriented smooth 1-submanifold of the 3-sphere $S^3$---denote by $\chi_4(L)$ the largest integer among the Euler characteristics of smooth oriented surfaces in the $4$-ball $B^4$ without closed components and oriented boundary $L\subset \partial B^4=S^3$. In particular, for a knot $K$---a connected link---one has
%\begin{equation}\label{eq:2g_4=1-chi}
$2g_4(K)=1-\chi_4(K)$,
%\end{equation}
where $g_4$ denotes the \emph{slice genus}.
The \emph{slice-Bennequin inequality} states that%, for all integers $n\geq 2$,
\begin{equation}\label{eq:sBineqwr}
|\wr(\beta)|\leq -\chi_4\left(\widehat{\beta}\right)+n\quad \text{%for all integers $n\geq 1$ and
for all }\beta\in B_n\quad
\text{\cite%[Slice-Bennequin Inequality]
{rudolph_QPasObstruction%}\cite{
,KronheimerMrowka_GenusofEmb}},
\end{equation}
where $\wr\colon B_n\to \Z$ denotes the \emph{writhe}, the group homomorphism with $\wr(a_i)=1$, and $\widehat{\beta}$ denotes the link obtained as the closure of $\beta$.
For $\beta$ with closure a knot, \eqref{eq:sBineqwr} reads
$\left|\wr(\beta)\right|\leq 2g_4\left(\widehat{\beta}\right)+n-1$.

One may wonder which other maps $f\colon B_n\to \R$ satisfy a similar inequality. Concretely, \cite[Question~1.6]{HKKMPRTT} asks whether, for each $n\geq 3$, there exist constants $A(n)$ and $C(n)$ such that
$\fd(\beta)\leq A(n)g_4\left(\widehat{\beta}\right)+C(n)$ for all $\beta\in B_n$ with closure a knot. %In fact, it is asked whether $A$ and $B$ can be chosen to be $2$ and $n-2$, respectively.
%We answer by establish the following.
We answer affirmatively with $A(n)$ independent of $n$; concretely, $A(n)=2$, which is optimal (e.g.~by the examples from~\cite[Prop.~4.7]{HKKMPRTT}).
\begin{thm}\label{thm:sBineq} For all integers $n\geq 3$, we have that
the FDTC $\fd\colon B_n\to \R$ satisfies \[|\fd(\beta)|\leq -\chi_4\left(\widehat{\beta}\right)+n\text{ for all }\beta\in B_n.\]
\end{thm}
%Since $2g_4(K)=1-\chi_4(K)$ for knots $K$,
Theorem~\ref{thm:sBineq} provides the affirmative answer claimed above, since it
%.\ref{item:A})
 reads
$|\fd(\beta)|\leq 2g_4\left(\widehat{\beta}\right)+n-1$ for all $\beta\in B_n$ such that $\widehat{\beta}$ is a knot.

While the result also holds for $n=2$, we do not include it in the statement since, for $\beta\in B_2$, one has
$\left|2\fd(\beta)\right|=
\left|\wr(\beta)\right|
\overset{\text{\eqref{eq:sBineqwr}}}{\leq}{-\chi_4\left(\widehat{\beta}\right)+n}$, where \eqref{eq:sBineqwr} is only applied for the rather elementary case of closures of elements in $B_2$.

%Of course Theorem~\ref{thm:CharPropForFDTC} is most interesting if, in contrast to the writhe and the FDTC, not every homogeneous quasimorphism satisfies a version of the slice-Bennequin inequality. In other words, there ought to exist a homogeneous quasimorphism $f$ such that there exist no constants $A$ and $B$ for which $f(\beta)\leq Ag_4\left(\widehat{\beta}\right)+C$ holds for all $\beta$ with $\widehat{\beta}$ a knot. We suspect that no such constants $A$ and $B$ exist for many homogeneous quasimorphism (even most, in some sense), but can currently not make this into a precise statement.

We provide more context for Theorem~\ref{thm:sBineq} in Section~\ref{sec:context}.
%, some of which might be taken by the reader as an indication that, a priori, Theorem~\ref{thm:sBineq} might be sightly surprising.
The main input in the proof of Theorem~\ref{thm:sBineq} is that the FDTC can be expressed in terms of the so-called homogenization of an instance of Upsilon. Here Upsilon is a knot invariant, introduced in~\cite{OSS_2014}, that has being a lower bound for the slice genus as a key feature; see Section~\ref{sec:proofsBi}.

%In the appendix, we provide a setup that unifies the  slice-Bennequin inequality~\eqref{eq:sBineqwr} and Theorem~\ref{thm:sBineq} as instances of the same homogenization phenomenon.

\subsection*{Comparing bounds on the defect and affine linear bounds for the slice genus}
At this point the reader may have wondered why Theorem~\ref{thm:CharPropForFDTC} and Theorem~\ref{thm:sBineq} appear in the same text, given they are concerned with different aspects of the FDTC. A link between these aspects %(i.e.~defect and an affine bound for the slice genus)
is provided in Proposition~\ref{prop:Df<=A(n-1)} below.

We ask whether satisfying the slice-Bennequin inequality %(together with the normalization conditions~\eqref{item:norm} and~\eqref{item:Bn-1=0})
as described in Theorem~\ref{thm:sBineq} characterizes the FDTC in the same way that having defect 1 does characterize the FDTC by 
Theorem~\ref{thm:CharPropForFDTC}.
\begin{Question}\label{q:CharFDTCviasB}Fix an integer $n\geq 3$.
Is the FDTC the unique homogeneous quasimorphism $\fd\colon B_n\to \R$ that satisfies $\fd\left(\Delta^2\right)=1$, $\fd(\beta)=0$ for all $\beta \in B_{n-1}\subset B_n$, and for which there exists a constant $C$ such that
\[|\fd(\beta)|\leq 2g_4\left(\widehat{\beta}\right)+C%n-1
\text{ for all $\beta\in B_n$ for which $\widehat{\beta}$ is a knot}
%|\fd(\beta)|\leq -\chi_4\left(\widehat{\beta}\right)+n\text{ for all }\beta\in B_n
?\]
\end{Question}
While we are unable to answer this question, we provide the following connection between the defect of a quasimorphism $f$ and the possible slopes of affine linear bounds for $g_4\left(\widehat{\beta}\right)$ in terms of $f(\beta)$. 
\begin{prop}\label{prop:Df<=A(n-1)} Fix $n\geq 3$ and 
let $f\colon B_n\to\R$ be a homogeneous quasimorphism.
 %that satisfies $f\left(B_{n-1}\right)=\{0\}$.
 If there exist constants $A,C\in\R$ such that
\[\left|f(\beta)\right|\leq Ag_4\left(\widehat{\beta}\right)+C\text{ for all $\beta\in B_n$ with closure a knot},
\]
then the defect $D_f$ of $f$ satisfies $%|f(\Delta^2)|\leq
D_f\leq \frac{A}{2}(n-1)$.
\end{prop}
In a first version of this article, the conclusion of Proposition~\ref{prop:Df<=A(n-1)} stated  $D_f\leq A(n-1)
$. The factor $1/2$ improvement
%(based on the fact that the stable commutator length of a commutator is at most $1/2$)
was pointed out by Tetsuya Ito.

Proposition~\ref{prop:Df<=A(n-1)} can be understood as a first step towards affirmatively answering Question~\ref{q:CharFDTCviasB}. Concretely, we have the following corollary.

\begin{corollary}\label{cor:defectvssB} Fix $n\geq 3$ and 
let $f\colon B_n\to\R$ be a homogeneous quasimorphism.
 that satisfies $f\left(B_{n-1}\right)=\{0\}$ and $f(\Delta^2)=1$.
 If there exist constants $A,C\in\R$ such that
$\left|f(\beta)\right|\leq Ag_4\left(\widehat{\beta}\right)+C\text{ for all $\beta\in B_n$ with closure a knot},
$
then $A>\frac{2}{n-1}$.
\end{corollary}
\begin{proof}[Proof of Corollary~\ref{cor:defectvssB}] Assume towards a contradiction that $A\leq \frac{2}{n-1}$. Then $D_f\leq \frac{A}{2}(n-1)\leq 1$ by Proposition~\ref{prop:Df<=A(n-1)}, hence $f=\fd$ by Theorem~\ref{thm:CharPropForFDTC}. However, for $f=\fd$, we have $A\geq 2$, e.g.~by the examples from~\cite[Prop.~4.7]{HKKMPRTT}.
\end{proof}
We do not expect Proposition~\ref{prop:Df<=A(n-1)} to be optimal. If the inequality for $D_f$ in Proposition~\ref{prop:Df<=A(n-1)} can be strengthened, concretely, if the next question can be answered affirmatively, then Theorem~\ref{thm:CharPropForFDTC} implies that Question~\ref{q:CharFDTCviasB} can be answered affirmatively (by arguing as in the proof of Corollary~\ref{cor:defectvssB}).
\begin{Question}\label{q:Df<=A/2} Fix $n\geq 3$ and 
let $f\colon B_n\to\R$ be a homogeneous quasimorphism
 that satisfies $f\left(B_{n-1}\right)=\{0\}$. If there exist constants $A,C\in\R$ such that
\[\left|f(\beta)\right|\leq Ag_4\left(\widehat{\beta}\right)+C\text{ for all $\beta\in B_n$ with closure a knot},
\]
does the defect $D_f$ of $f$ satisfies $D_f\leq A/2$?
\end{Question}

\subsection*{Ingredients for the proofs and structure of the paper}

In Section~\ref{sec:context}, we provide context for Theorem~\ref{thm:sBineq}.

In Section~\ref{sec:ProofThm1} we establish Theorem~\ref{thm:CharPropForFDTC}. The main step in the proof of Theorem~\ref{thm:CharPropForFDTC} is to show that every braid $\beta\in B_n$ that can be written as a braid word that contains at most $l$ occurrences of $a_1$ and no $a_1^{-1}$, can be decomposed as a particular product of full twists $\Delta^2$ and at most $l$ braids that are conjugate to braids in $B_{n-1}\subset B_n$. % In fact, the length of this decomposition is essentially optimal, which is reflected in the optimality statement of Theorem~\ref{thm:CharPropForFDTC}.
%As mentioned above, any homogeneous quasimorphism on the braid group (and in fact a mapping class group of a surface with a marked boundary component) that satisfies an appropriate positivity condition is equal to the FDTC up to a normalization factor. We discuss this by generalizing Proposition~\ref{prop:CharPropForFDTC} to all mapping class groups. In fact, these results can be phrased rather generally for groups with a certain notion of order on them; compare Appendix~\ref{ap:qmviafloors}.

In Section~\ref{sec:BF}, we show that Proposition~\ref{prop:uncqms} is a rather immediate consequence of  the work on group actions on $\delta$-hyperbolic spaces that satisfy WPD (weak proper discontinuity)~\cite{BestvinaFujiwara_02}. 
% These techniques are of a rather different flavour than the rest of the text, which is why their appearance is confined to Section~\ref{sec:BF}.

In Section~\ref{sec:proofsBi}, we establish Theorem~\ref{thm:sBineq}. A key ingredient is the reinterpretation of the FDTC as a linear combination of the writhe and the homogenization of $\Upsilon(t)$~\cite[Theorem~1.3]{feller_hubbard}. Additionally, we use some facts about concordances between braid closures. The latter is also what we use in the proof of Proposition~\ref{prop:Df<=A(n-1)}.
%, which we also provide in Section~\ref{sec:proofsBi}.

We conclude the paper with a perspective that allows to view both the slice-Bennequin~\eqref{eq:sBineqwr} and Theorem~\ref{thm:sBineq} as instances of an observation concerning the homogenization of concordance homomorphisms; see Appendix~\ref{ap:sBiforI}.

%In Appendix~\ref{ap:qmviafloors}, we discuss uniqueness results for homogeneous quasimorphisms that satisfy a certain positivity condition. In particular, we show that such results hold rather generally and seem less braid groups specific than Theorem~\ref{thm:CharPropForFDTC}.  As an aside this yields that many different definitions of the FDTC on mapping class groups agree.

\subsection*{Acknowledgements} The author thanks Jonathan Bowden, Francesco Fournier-Facio, and Alessandro Sisto for helpful conversations concerning the construction of quasimorphisms via actions on $\delta$-hyperbolic spaces that satisfy WPD. The author also thanks Tetsuya Ito for a stimulating exchange and an improvement concerning Proposition~\ref{prop:Df<=A(n-1)}.
Finally, the author gratefully acknowledges support by the SNSF Grant~181199.

\section{Context for Theorem~\ref{thm:sBineq}: Bennequin and slice-Bennequin inequalities}\label{sec:context}
For simplicity of exposition, in this section all braids $\beta$ are assumed to have a knot as their closure. All that is said translates to the general setup if $g_k$ is replaced by $-(\chi_k-1)/2$ for $k\in\{3,4\}$.

\subsection*{The Bennequin inequality}
The slice-Bennequin inequality
\[\wr(\beta)\leq 2g_4\left(\widehat{\beta}\right)+n-1\text{ for all }\beta\in B_n\text{ \cite{rudolph_QPasObstruction,KronheimerMrowka_GenusofEmb}},\]
was pre-dated by the Bennequin inequality
\[\wr(\beta)\leq 2g_3\left(\widehat{\beta}\right)+n-1\text{ for all }\beta\in B_n \text{ \cite{Bennequin_Entrelacements}},\]
where $g_3$ denotes the smallest genus among smooth surface in $S^3$ with boundary the knot.
There is a conceptual gap between these two results. The slice-Bennequin inequality needed a strong input from smooth $4$-manifold theory the so-called local-Thom conjecture as proven by Kronheimer and Mrowka~\cite{KronheimerMrowka_GenusofEmb}. The `smooth' is crucially here. Indeed, the analog statement in the locally-flat setting (where $g_4$ is replaced with the topological slice genus) is well-known to fail in many instances; see e.g.~\cite{rudolph_QPasObstruction,Rudolph_84_SomeTopLocFlatSurf,BaaderFellerLewarkLiechti_15}.

Concerning the FDTC, we point to a version of the Bennequin inequality by Ito~\cite[Theorem~1.2]{Ito_11}:
$\fd(\beta)< 2g_3\left(\widehat{\beta}\right)\frac{2}{n+2}-\frac{2}{n+2}+\frac{3}{2}$ for all $\beta\in B_n$.% with $\widehat{\beta}$ a knot.
\footnote{We note that \cite[Theorem~1.2]{Ito_11} is phrased for the Dehornoy floor (compare Section~\ref{sec:ProofThm1}). However, since the argument only uses properties of the Dehornoy floor that are also satisfied by the FDTC, it translates to a statement about the FDTC.}
Meaning that the Bennequin inequality holds with a slope $A(n)$ for which it is known not to hold when $g_3$ is replaced by $g_4$.

\subsection*{The slice-Bennequin inequality for (quasi-)positive braids}

 %recall point to the fact that
Recall that the key input for Rudolph's proof of the slice-Bennequin inequality~\eqref{eq:sBineqwr} is that it holds (in fact with equality) for positive braids with closure a torus knot (by the local Thom conjecture~\cite{KronheimerMrowka_GenusofEmb}) and hence, as observed by Rudolph, also for positive (actually also quasipositive) braids.
Then~\eqref{eq:sBineqwr} follows using that, for all $\beta\in B_n$ and generators $a_i$, 
$\wr(\beta a_i)-\wr(\beta)\geq 1$ and there exists a cobordism with Euler characteristic $-1$ between the closures of $\beta a_i$ and $\beta$.
For the FDTC,
combining $\fd(\beta)\leq \wr(\beta)-1$ for quasipositive braids $\beta\neq 1$
(this follows readily by using that $\omega(a_i)=0$, which is implied by $\fd(B_n-1)=\{0\}$, and that $\fd$ is homogeneous quasimorphism of defect at most 1)
%\footnote{Reason: one can remove $\wr(\beta)-1$ generators from a non-trivial quasipositive $\beta$ to find a braid conjugate to $a_1$, $\fd(a_1)=0$, and
%$\forall\alpha\in B_n:|\fd(a_i\alpha)-\fd(\alpha)|\leq 1$), hence $\wr(\beta)=-\chi_4\left(\widehat{\beta}\right)+n$ (i.e.~\eqref{eq:2g_4=1-chi}) for quasipositive braids yields~\eqref{eq:sliceBennequinfdtcqp}; see~\cite[Proof of Theorem~1.5]{HKKMPRTT}).}
with~\eqref{eq:sBineqwr} yields
\begin{equation}\label{eq:sliceBennequinfdtcqp}
\fd(\beta)\leq -\chi_4\left(\widehat{\beta}\right)+n-1\quad \text{for all quasipositive $\beta\neq 1$ in $B_n$;}
\end{equation} see \cite[Theorem~1.5]{HKKMPRTT}. However, the strategy of reducing to the statement for positive (or quasipositive) braids cannot be carried over to establish the slice-Bennequin inequality for the FDTC since $\fd(\beta a_i)-\fd(\beta)\geq 1$ does not hold in general. %(as is e.g.~evident from~\eqref{eq:sliceBennequinfdtcqp}).
Also, %this argument for~\eqref{eq:sliceBennequinfdtcqp} is rather specific for quasipositive $\beta$, since
there is no bound on $\fd(\beta)$ in terms of an expression depending on $\wr(\beta)$ that holds for all braids $\beta\in B_n$.

%\subsection*{Proof via homogenization of the concordance homomorphism Upsilon}
%Some of the above might be taken by the reader as an indication that, a priori, Theorem~\ref{thm:sBineq} (in particular with slope $A(n)=2$), might be sightly surprising.
%The key to ruin the surprise and the main input in the proof of Theorem~\ref{thm:sBineq} is that the FDTC can be expressed in terms of the so-called homogenization of an instance of Upsilon; see Section~\ref{sec:proofsBi}.

\subsection*{Does a version of the slice-Bennequin inequality hold for all quasimorphisms?}
In light of the writhe and FDTC satisfying the slice-Bennequin inequality, one might wonder whether,
for each braid group $B_n$ with $n\geq 3$, there exists a quasimorphism $f\colon B_n\to\R$ which does not satisfy a version of the slice-Bennequin inequality. We state this precisely.
\begin{question}\label{q:qmwithoutSBI}
Fix $n\geq 3$. Does there exist a quasimorphism $f\colon B_n\to\R$ such that, for every $A,C\in \R$, there exists a $\beta\in B_n$ with closure a knot such that $|f(\beta)|>Ag_4\left(\widehat{\beta}\right)+C$?
\end{question}
We note that since every quasimorphism is at bounded distance of a homogeneous one, asking the question for homogeneous quasimorphisms or quasimorphisms amounts to the same.

We suspect that many (if not most) quasimorphisms do not satisfy a version of the slice-Bennequin inequality; why, after all, would there be such a connection between quasimorphisms and concordance.
For example, we suspect that many Brooks-like quasimorphisms as provided in~\cite{BestvinaFujiwara_02} (a construction that, as far as we know, is devoid of connections to concordance) are quasimorphisms as asked for in Question~\ref{q:qmwithoutSBI}. However, we are not able to confirm this at this point.

To answer Question~\ref{q:qmwithoutSBI} in the positive, the reader might be tempted (the author certainly was) to construct a homogeneous quasimorphism $f\colon B_n\to \R$ that is non-zero on some braid $\beta$ with the property that $\lim_{k\to\infty} \frac{\chi_4\left(\widehat{\beta^k}\right)}{k}=0$.
Indeed, one readily checks that such an $f$ is as asked for in Question~\ref{q:qmwithoutSBI}. However, the only braids we are able to find with the property $\lim_{k\to\infty} \frac{\chi_4\left(\widehat{\beta^k}\right)}{k}=0$
are braids that are (up to taking conjugates) of the form $\alpha\overline{\alpha^{-1}}$ for some braid $\alpha$, and a small calculation reveals that, for all $\alpha\in B_n$, all homogeneous quasimorphisms $f\colon B_n\to\R$ vanish on $\alpha\overline{\alpha^{-1}}$.
(Here, $\overline{\gamma}\in B_n$ is the result of changing all $a_i^{\pm 1}$ in a braid word for $\gamma\in B_n$ to $a_{n-i}^{\pm 1}$.)
\section[The proof of Theorem~\ref{thm:CharPropForFDTC}]
{The proof of Theorem~\ref{thm:CharPropForFDTC}}\label{sec:ProofThm1}
\subsection*{Definition of the FDTC via the Dehornoy order} We fix an integer $n\geq 2$.
A braid $\beta$ is said to be \emph{Dehornoy positive}, denoted by $\beta\succ_\textrm{Deh} 1$, if it can be written as a braid word that, for some integer $1\leq i<n$, contains a braid generator $a_i$ but no $a_i^{-1}$ or any generators $a_j^{\pm 1}$ for $j<i$. We write $\beta\succeq_\textrm{Deh} 1$ if $\beta\succ_\textrm{Deh} 1$ or $\beta=1$. Dehornoy showed that this gives a well-defined left-invariant total order $\succeq_\textrm{Deh}$ on $B_n$ by setting $\beta\succeq_\textrm{Deh}\alpha$ to mean $\alpha^{-1}\beta\succeq_\textrm{Deh}1$~\cite{Dehornoy_94_Braidgroups}.
The \emph{Dehornoy floor} $\lfloor\beta\rfloor$ is the unique integer $m$ such that $(\Delta^2)^{m+1}\succ_\textrm{Deh}\beta\succeq_\textrm{Deh}(\Delta^2)^{m}$. For any $\beta\in B_n$, its \emph{fractional Dehn twist coefficient} is $\omega(\beta)\coloneqq\lim_{k\to\infty}\frac{\lfloor\beta^k\rfloor}{k}$; see~\cite{Malyutin_Twistnumber}. In other words, $\fd$ equals the homogenization of the Dehornoy floor. We refer to~\cite{Malyutin_Twistnumber} for more details on this approach to the FDTC and how one derives its properties (e.g.~being a homogeneous quasimorphism with defect at most~1).

\begin{Remark}\label{rmk:fd>0impliesDpostive}
It is essentially immediate from this definition that, if $\fd(\beta)>0$, then $\beta$ can be written as a braid word with at least one $a_1$ and no $a_1^{-1}$ (and, in particular, $\beta\succ 1$).
Indeed, if $\beta$ can be written as a braid word without $a_1$ or $a_1^{-1}$,
then $(\Delta^2)\succ_\textrm{Deh}\beta\succ_\textrm{Deh}(\Delta^2)^{-1}$. Hence $1\geq \lfloor\beta^k\rfloor\geq -1$, which implies $\fd(\beta)=0$. Therefore, $\beta$ can be written as braid word that either contains only $a_1$ or only $a_1^{-1}$. If it were the latter, then $1\succ_\textrm{Deh}\beta^k$, hence $0\geq\lfloor\beta^k\rfloor$, which implies
$0\geq\fd(\beta)$. %We conclude that $\beta$ can be written as braid word that contains only $a_1$.
\end{Remark}

\subsection*{Proof of Theorem~\ref{thm:CharPropForFDTC}}
The proof below %is elementary in the following sense. It only
uses relations in the braid group, general properties of homogeneous quasimorphisms, and the property of the FDTC discussed in Remark~\ref{rmk:fd>0impliesDpostive}.  

%\begin{prop}[\cite{Malyutin_Twistnumber}]\label{prop:CharPropForFDTC}For every $n\geq2$,
%there exists exactly one homogenous quasimorphism
%$\fd\colon B_n \to \R$ such that
%\begin{enumerate}
%  \item $\fd(\Delta^2)=1$.
%  \item $\fd(\beta)=0$ for all $\beta \in B_{n-1}\subset B_n$.
%  \item\label{item:pos} (positivity) $\fd(\beta)\geq 0$ for all $\beta\in B_n$ with $\beta \succeq 1$.
%\end{enumerate}
%\end{prop}
 %We recall the proof of the proposition, as we will return to this proof in the appendix.

\begin{proof}[Proof of Theorem~\ref{thm:CharPropForFDTC}]
We fix $n\geq 3$, and % and establish Theorem~\ref{thm:CharPropForFDTC} by checking the conditions of Proposition~\ref{prop:CharPropForFDTC}.
let $\d\colon B_n\to \R$ be any homogeneous quasimorphism that satisfies the assumptions.
%~\eqref{item:norm} and~\eqref{item:Bn-1=0} of Theorem~\ref{thm:CharPropForFDTC}.
%We %first establish
%show that $\d=\fd$.

Assume towards a contradiction that $\d\neq\fd$. Pick $\beta_\textrm{w}\in B_n$ with $\fd(\beta_w)-\d(\beta_w)\neq 0$.
Since $\d$ and $\fd$ are homogeneous, there exists %an integer
$k_1\in\Z$ such that
\[\fd\left(\beta_\textrm{w}^{k_1}\right)-\d\left(\beta_\textrm{w}^{k_1}\right)=k_1(\fd(\beta_w)-\d(\beta_w))>1.\]
Since $f(ab)=f(a)+f(b)$ for all homogeneous quasimorphisms $f\colon G\to \R$ and commuting $a,b\in G$, and since $\fd\left(\Delta^2\right)=\d\left(\Delta^2\right)=1$, there exists $k_2\in\Z$
such that \[\fd\left(\beta_\textrm{w}^{k_1}(\Delta^2)^{k_2}\right)=\fd\left(\beta_\textrm{w}^{k_1})+k_2>0>\d(\beta_\textrm{w}^{k_1}\right)+k_2=\d\left(\beta_\textrm{w}^{k_1}(\Delta^2)^{k_2}\right).\]
We define $\beta\coloneqq \beta_\textrm{w}^{k_1}(\Delta^2)^{k_2}$. Since $\fd(\beta)>0$,
%By Proposition~\ref{prop:CharPropForFDTC} it suffices to check that $\d(\beta)\geq 0$ for all $\beta\in B_n$ with $\beta \succeq 1$.
%
%Since Let $\beta$ be $n$-braid with $\beta\succeq 1$. If $\d(\beta)=0$, we are done. Hence we consider %only $\beta$ such that $\d(\beta)\neq 0$.
%Since $\d(\beta)\neq 0$, $\beta$ can not be written as a braid word without $a_1$ and $a_1^{-1}$ (indeed, %otherwise it would be an element of $B_{n-1}$ for which $\d$ is 0 by~\eqref{item:Bn-1=0}).
%Hence,
we have that $\beta$ can be given by a braid word with no occurrences of $a_1^{-1}$ but at least one $a_1$ by Remark~\ref{rmk:fd>0impliesDpostive}.
%\footnote{As an aside, we note that the proof of Theorem~\ref{thm:CharPropForFDTC} upto this point reproves Proposition~\ref{prop:CharPropForFDTC}. Indeed, we have $\beta\succ 1$; in particular, $\beta\succeq 1$, which contradicts $\d(\beta)<0$.}

We proceed by showing $\d(\beta)\geq 0$, which contradicts $\d(\beta)<0$.
% since $\beta\succ1$.
We may and do assume that the number of occurrences of $a_1$ in the braid word without $a_1^{-1}$ we picked for $\beta$ is even. (Indeed, otherwise we consider $\beta\beta$, which also satisfies $0>\d(\beta\beta)$ since $\d(\beta\beta)=\d(\beta)+\d(\beta)$.)
%\succ1$ and $\d(\beta\beta)\neq 0$ by left-invariance of $\succ$ and homogenity of $\d$, respectively.)
Hence we have that
\[\beta=\prod_{i=1}^{2l}(a_1\beta_i),\] where $l$ is a positive integer and the $\beta_i$ are (possibly trivial) $n$-braids in $B_{n-1}\subset B_n$.

Next we observe that $\beta$ maybe conjugated to a braid of the form $\Delta^{2l}\prod_{i=1}^{l}L_{i}R_{i}$, where $R_i$ is an $n$-braid given by a braid word without $a_1^{\pm1}$ (in other words an element of $B_{n-1}\subset B_n$) and $L_i$ an $n$-braid given by a braid word without $a_{n-1}^{\pm1}$. To describe $L_i$ and $R_i$, we consider the element
\[\Delta\coloneqq\prod_{i=1}^{n-1}a_1a_2\cdots a_{n-i}=\prod_{i=1}^{n-1}a_{n-1}a_{n-2}\cdots a_{i+1}a_{i}\in B_n,\] known as the (positive) half-twist since $\Delta\Delta=\Delta^2$. We denote by $\Delta_L\in B_n$ and $\Delta_R\in B_n$ the half-twist on the first $n-1$ strands and the last $n-1$ strands, respectively. In other words, $\Delta_R$ is the image of the half twist $\Delta\in B_{n-1}$ under the inclusion $\iota\colon B_{n-1}\to B_n,a_i\mapsto a_{i+1}$, while $\Delta_L$ is the image of the half twist $\Delta\in B_{n-1}$ under the inclusion $a_i\mapsto a_i$. We also denote by $\overline{\beta}$ the braid obtained from $\beta$ by replacing $a_i^{\pm1}$ with $a_{n-i}^{\pm1}$, and recall that $\Delta^{\pm1}\beta=\overline{\beta}\Delta^{\pm1}$. With this we see
\begin{align*}
a_1\beta_{2i-1}a_1\beta_{2i}&=\Delta^2\Delta^{-2}a_1\beta_{2i-1}a_1\beta_{2i}\\
&=\Delta^2\underbrace{\Delta^{-1}}_{\hspace{-2cm}\Delta_R^{-1}\left(a_1^{-1}a_2^{-1}\cdots a_{n-2}^{-1}a_{n-1}^{-1}\right)\hspace{-2cm}}
a_{n-1}\overline{\beta_{2i-1}}\overbrace{\Delta^{-1}}^{\hspace{-2cm}\Delta_L^{-1}\left(a_{n-1}^{-1}a_{n-2}^{-1}\cdots a_{2}^{-1}a_1^{-1}\right)\hspace{-2cm}}a_1\beta_{2i}\\
&=\Delta^2\Delta_R^{-1}\left(a_1^{-1}a_2^{-1}\cdots a_{n-2}^{-1}\right)\overline{\beta_{2i-1}}\Delta_L^{-1}\left(a_{n-1}^{-1}a_{n-2}^{-1}\cdots a_{2}^{-1}\right)\beta_{2i}.
\end{align*}
Hence,
\begin{align*}
\beta&=\prod_{i=1}^{l}\Delta^2\Delta_R^{-1}\left(a_1^{-1}a_2^{-1}\cdots a_{n-2}^{-1}\right)\overline{\beta_{2i-1}}\Delta_L^{-1}\left(a_{n-1}^{-1}a_{n-2}^{-1}\cdots a_{2}^{-1}\right)\beta_{2i}\\
&=\Delta_R^{-1}\left(\prod_{i=1}^{l}\Delta^2\left(a_1^{-1}a_2^{-1}\cdots a_{n-2}^{-1}\right)\overline{\beta_{2i-1}}\Delta_L^{-1}\left(a_{n-1}^{-1}a_{n-2}^{-1}\cdots a_{2}^{-1}\right)\beta_{2i}\Delta_R^{-1}\right)\Delta_R\\
&=\Delta_R^{-1}\left(\Delta^{2l}\prod_{i=1}^{l}\left(a_1^{-1}a_2^{-1}\cdots a_{n-2}^{-1}\right)\overline{\beta_{2i-1}}\Delta_L^{-1}\left(a_{n-1}^{-1}a_{n-2}^{-1}\cdots a_{2}^{-1}\right)\beta_{2i}\Delta_R^{-1}\right)\Delta_R,
\end{align*}
meaning that $\beta$ is conjugate to 
\[\beta'=\Delta^{2l}\prod_{i=1}^{l}\underbrace{\left(a_1^{-1}a_2^{-1}\cdots a_{n-2}^{-1}\right)\overline{\beta_{2i-1}}\Delta_L^{-1}}_{L_i}\underbrace{\left(a_{n-1}^{-1}a_{n-2}^{-1}\cdots a_{2}^{-1}\right)\beta_{2i}\Delta_R^{-1}}_{R_i}.\]

Finally, using that homogeneous quasimorphisms are constant on conjugation classes, $\d(\Delta^2\alpha)=1+\d(\alpha)$ (since $\d(\alpha\beta)=\d(\alpha)+\d(\beta)$ for commuting $\alpha$ and $\beta$ and a homogeneous quasimorphism $\d$), and
$\d(\alpha\beta\gamma)\geq \d(\alpha\gamma)+\d(\beta)-1$ for all $n$-braids $\alpha$, $\beta$, $\gamma$ (which follows from $\d$ having defect at most $1$ and being constant on conjugation classes), we calculate   
\begin{align*}
\d(\beta)&=\d(\beta')
=\d\left(\Delta^{2l}\prod_{i=1}^{l}L_iR_i\right)
=l+\d\left(\prod_{i=1}^{l}L_iR_i\right)
\\&\geq l+\d(L_1)+\d\left(R_1\prod_{i=2}^{l}L_iR_i\right)-1
%\\&\geq l+\d(L_1)+\d(L_2)+\d\left(R_1R_2\prod_{i=3}^{l}L_iR_i\right)-1-1
\\&\cdots
\\&\geq l+\d(L_1)+\d(L_2)+\cdots+\d(L_l)+\d\left(\prod_{i=1}^{l}R_i\right)-l
\\&=\d(L_1)+\d(L_2)+\cdots+\d(L_l)+\d\left(\prod_{i=1}^{l}R_i\right).
\end{align*}
Since $\d$ vanishes on braids that can be written without $a_1^{\pm1}$ (which include $\prod_{i=1}^{l}R_i$), and thus (by conjugation invariance) also on braids without $a_{n-1}^{\pm1}$ (which include $L_i$), we have $\d(\beta)\geq 0>\d(\beta)$. 
\end{proof}
The defect of $\fd$ is known to be $1$ for $n\geq 3$. We provide an argument, which is of the same flavour (but much simpler) than the above proof, and
\begin{lemma}\label{lemma:qmhaveD>=1} For $n\geq 3$
if a homogeneous quasimorphism $f\colon B_n\to\R$ satisfies %$\d(\Delta^2)=1$ and
$f\left(B_{n-1}\right)=\{0\}$, then
the defect of $f$ is bounded below by $\left|f\left(\Delta^2\right)\right|$, i.e.~$\left|f\left(\Delta^2\right)\right|\leq D_f$.
\end{lemma}
\begin{proof}
%Let $f$  be a homogeneous quasimorphism $B_n\to\R$ with
%$f(B_{n-1})=0$ and
%let $D$ denote its defect.
First we note that
\[f\left(a_1a_2\cdots a_{n-2}a_{n-1}a_{n-1}a_{n-2}\cdots a_2a_1\right)=f(\Delta^2\Delta_R^{-2})=f(\Delta^2)+f(\Delta_R^{-2})=f(\Delta^2),\]
where the first equality is due to equality of the braids, the second equality uses that $\Delta^2$ is in the center, and the last equality uses that $f$ vanishes on $\Delta_R\in B_{n-1}\subset B_n$.
Hence, for $\alpha=a_2\cdots a_{n-2}a_{n-1}a_{n-1}a_{n-2}\cdots a_2\in B_{n-1}$ and $\beta=a_{1}a_{1}$, we find \[D_f\geq \left|f(\alpha\beta)-f(\alpha)-f(\beta)\right|=\left|f(\Delta^2)-0-0\right|,\] where we used that $f$ evaluates to the same on $a_1a_2\cdots a_{n-2}a_{n-1}a_{n-1}a_{n-2}\cdots a_1$ and its conjugate $\alpha\beta$.
\end{proof}
As an aside we note that the proof of Lemma~\ref{lemma:qmhaveD>=1} shows that for $\fd$ the supremum $D_f$ is attained when $n\geq 3$.
\section{Constructions of quasimorphisms and the proof of Proposition~\ref{prop:uncqms}% and quasimorphisms with prescribed values
}\label{sec:BF}
In this section we discuss the existence of many homogeneous quasimorphism on $B_n$ for $n\geq 3$ as claimed in Proposition~\ref{prop:uncqms}.
%We also discuss the existence of homogeneous quasimorphism with a prescribed value on a given braid.
%The latter allows us to produce an example of a quasimorphism that does not satisfy a version of the slice-Bennequin inequality.
%\subsection*{The proof of Proposition~\ref{prop:uncqms}}
%In this section we prove Proposition~\ref{prop:uncqms}
We make use of a geometric group theory setup due to Bestvina and Fujiwara~\cite{BestvinaFujiwara_02}, which we do not recall in detail. Since this makes this section the least self-contained, we point out that %this section is completely independent from those that follow. In particular,
skipping this section can be done at no cost of understanding the results from the introduction except, of course, Proposition~\ref{prop:uncqms}.
 
Proposition~\ref{prop:uncqms} reduces to the following lemma.
\begin{lemma}\label{lem:l1embedsinQH} Let $n\geq 3$. There exist an injective $\R$-linear map
\[%L\colon
\ell^1\to \left\{f\colon B_n\to\R\mid \text{$f$ is a homogeneous quasimorphism and } f(B_{n-1})=\{0\}\right\}.\]
\end{lemma}
Here $\ell^1$ denotes the vector space of real-valued sequences $\{a_i\}_{n\in N}$ with $\sum_{n=1}^\infty |a_i|<\infty$. 
Dropping the condition $f(B_{n-1})=\{0\}$, Lemma~\ref{lem:l1embedsinQH} is known by work of Bestvina and Fujiwara. Indeed,
there exist an injective $\R$-linear map \[%L\colon
\ell^1\to \{f\colon B_n\to\R\mid \text{$f$ is a homogeneous quasimorphism}\}\]
by~\cite[Theorem~7 and Proposition~11]{BestvinaFujiwara_02}. In fact, inspection of their proof reveals that all elements in the image of the $\R$-linear map they construct vanish on $B_{n-1}$. We explain this using the setup, notations, and results from~\cite{BestvinaFujiwara_02}. We only make use of these in the proof of Lemma~\ref{lem:l1embedsinQH}, and we only invoke Lemma~\ref{lem:l1embedsinQH} to prove Proposition~\ref{prop:uncqms}. 

\begin{proof}[Proof of Lemma~\ref{lem:l1embedsinQH}] Bestvina and Fujiwara construct a large vector subspace of the vector space of homogeneous quasimorphism on a group $G$ whenever the group $G$ has an action on a $\delta$-hyperbolic space $X$ that satisfies weak proper discontinuity (WPD for short)~\cite[Theorem~7]{BestvinaFujiwara_02}.
Actually, Bestvina and Fujiwara construct quasimorphisms that are in general not homogeneous, and then consider the quotient of the vector space of quasimorphisms by bounded functions. However, this quotient is readily identified with the vector space of homogeneous quasimorphisms. This identification is given by taking the quasimorphisms~$h_\omega$ from the construction of Bestvina and Fujiwara to their homogenizations $\widetilde{h_\omega}$. Under this identification, their construction translates to constructing a subspace of the vector space of homogeneous quasimorphism isomorphic to $\ell^1$ given by
$\{\sum_{n=1}^\infty a_ib_i\mid \sum_{n=0}^\infty|a_i|<\infty\}$, where the $b_i$ are elements of the form~$\widetilde{h_\omega}$.

From the construction in~\cite[Section 2]{BestvinaFujiwara_02} of the homogeneous quasimorphism ${h_\omega}$ it follows that if an element $r\in G$ has a fixed point $x_0\in X$, then the homogeneous quasimorphism $\widetilde{h_\omega}$ vanishes on $r$. Indeed, choosing $x_0$ as the basepoint in their construction of the quasimorphism $h_\omega$, we see that $h_\omega(r^k)=0$ for all $k\in \Z$. In particular, the homogeneous quasimorphism $\widetilde{h_\omega}$ satisfies $\widetilde{h_\omega}(r)\coloneqq\lim_{k\to\infty}\frac{h_\omega(r^k)}{k}=\lim_{k\to\infty}\frac{0}{k}=0$.

For technical reasons we choose our group $G$ to be the quotient $G\coloneqq B_n/\langle \Delta^2\rangle$ rather than $B_n$. Of course any quasimorphism on $G$ gives rise to one on $B_n$ by composing with the quotient map $\pi\colon B_n\to G$. Thus, by the last paragraph it remains to check that $G$ has an action on a $\delta$-hyperbolic space that satisfies WPD such that the elements of $\pi(B_{n-1})\subset G$ have a fixed point.
% it naturally acts on the curve and arc complex of the $n$-punctured disc.
To do this we identify $B_n$ with the mapping class group of the $n$-punctured disc and
we identify $G=B_n/\langle \Delta^2\rangle$ with a finite index subgroup of the mapping class group of the $(n+1)$-punctured sphere. Then $G$ naturally acts on the curve complex $X$ of the $(n+1)$-punctured sphere.
The curve complex $X$ is $\delta$-hyperbolic and the action of $G$ on $X$ satisfies WPD since the action of the full mapping class satisfies WPD~\cite[Proposition~11]{BestvinaFujiwara_02} and restricting an action that satisfies WPD to a finite index subgroup yields an action that satisfies WPD.
We conclude the proof by noting that there exists a simple closed curve $\gamma$ in the $(n+1)$-punctured sphere (in particular, $[\gamma]\in X$) such that $\pi(B_{n-1})=\{[\phi]\in G\mid [\phi][\gamma]=[\gamma]\}$. For sake of completeness, we describe such a $\gamma$ explicitly.

For this we make the identification of $B_{n}$ with the mapping class group of the $n$-punctured disc $D$ (taken as the closed unit disc in $\C$ with the punctures placed on the open interval $(-1,1)$ and ordered by the usual order on $(-1,1)\subset \R$) explicit. Namely, we chose an identification isomorphism that sends the generator $a_i$ to the mapping class given by a positive half-twist that exchanges the $i$-th and $(i+1)$-th punctures and is the identity outside a small neighbourhood of the arc on the real line connecting the $i$-th and $(i+1)$-th puncture. We further identify the $(n+1)$-punctured sphere with the quotient $D/S^1$, where the punctures are as for $D$ with one extra puncture: the point $\infty$ in the quotient corresponding to the collapsed $S^1$. This yields an explicit identification of $G$ with the subgroup of the mapping class group of the $(n+1)$-punctured sphere given by those mapping classes that fix the puncture $\infty$. This identification is such that the quotient map $\pi\colon B_n\to G$ is identified with the group homomorphism between the mapping class groups induced by the quotient map $D\to D/S^1$. See e.g.~\cite{HironakaEiko_06,Birman_74_BraidsLinksAndMCGs} for these identifications.

With this set up, we choose $\gamma$ to be a simple closed curve in $D\setminus S^1\subset D/S^1$ that is the boundary of a round disc in $D\setminus S^1$ that contains all but the first puncture. Then, indeed, $B_{n-1}\subset B_n$ is sent to mapping classes that have a representative that restricts to the identity on $\gamma$.
\end{proof}
\begin{proof}[Proof of Proposition~\ref{prop:uncqms}]
Fix $\epsilon>0$. And, for $r\in\R$, let $f_r$ be the image of the $r$-th basis element of a chosen basis for $\ell^1$ under an injective map guaranteed to exist by Lemma~\ref{lem:l1embedsinQH}. Up to multiplication with a constant, we can arrange for $f_r$ to satisfy $f_r(\Delta^2)\geq 0$ and $D_{f_r}<\epsilon$.
Define $g_r\coloneqq \frac{1}{1+f_r(\Delta^2)}(\omega+f_r)$, and note that $g_r(\Delta^2)=1$, $g_r(B_{n-1})=\{0\}$, and $D_{g_r}\leq 1+D_{f_r}<1+\epsilon$. Hence, for all but at most one $a\in\R$, $\{g_r\}_{r\in\R\setminus \{a\}}$ is a basis of a subspace of \[\left\{f\colon B_n\to\R\mid \text{$f$ is a homogeneous quasimorphism and } f(B_{n-1})=\{0\}\right\}.\qedhere\]
\end{proof}

\section{The proofs of Theorem~\ref{thm:sBineq} and Proposition~\ref{prop:Df<=A(n-1)}}\label{sec:proofsBi}
For the proof of Theorem~\ref{thm:sBineq}, we use that the FDTC can be expressed in terms of the homogenization of the Upsilon invariant. % as follows.
For all $\beta\in B_n$ and $t=\frac{2}{n-1}$, we have
\begin{equation}\label{eq:fdtcviaUpsilon}
\fd(\beta)=\frac{\widetilde{\Upsilon}_\beta(t)}{t}+\frac{\wr(\beta)}{2},
\end{equation} 
by \cite[Theorem~1.3]{feller_hubbard}. 
Here, for each $\beta\in B_n$ and for $\delta\coloneqq a_1a_2\cdots a_{n-1}\in B_n$,
\begin{equation}\label{eq:defhomUpsilon}
\widetilde{\Upsilon_\beta}\coloneqq \lim_{k\to\infty}\frac{\Upsilon_{\widehat{\beta^{nk}\delta}}(t)}{nk},\end{equation}
where %$\delta=a_1a_2\cdots a_{n-1}\in B_n$ and
for a knot $K$ and $t\in[0,1]$ we denote by $\Upsilon_K(t)$ the Upsilon invariant introduced in~\cite{OSS_2014}.
For more details on homogenization of knot invariants compare~\cite{GambaudoGhys_BraidsSignatures,Brandenbursky_11} and Appendix~\ref{ap:sBiforI}. For $\Upsilon$ specifically see~\cite{FellerKrcatovich_16_OnCobBraidIndexAndUpsilon}.
%, and for a somewhat general setup of homogenizations of knot invariants compare with~\cite[Appendix~A]{feller_hubbard}.

Recasting $\fd$ using $\Upsilon$ via~\eqref{eq:fdtcviaUpsilon} allows us to make use of the following slice genus bound. For every knot $K$, we have
\begin{equation}\label{eq:Upsilon<=tg4}
\Upsilon_K(t)\leq tg_4(K)\text{ for all }t\in[0,1]\quad\text{\cite[Theorem~1.11]{OSS_2014}}.
\end{equation} 

As a further input for the proof of Theorem~\ref{thm:sBineq}, but also the proof of Proposition~\ref{prop:Df<=A(n-1)}, we need cobordisms with small genera between knots and links arising as connected sums and arising as closures of compositions of braids. %Those are provided by the following lemma.
\begin{lemma}\label{lem:cobordisms}
Let $\alpha$, $\beta$, and $\gamma$ be in $B_n$.
\begin{enumerate}[(a)]
\item \label{item:coba}There exists a cobordism given by $(n-1)$ $1$-handles between $\widehat{\alpha\beta}$ and a
connected sum of $\widehat{\alpha}$ and $\widehat{\beta}$.
\item \label{item:cobb} If at least one of the braids $\alpha$, $\beta$, or $\gamma$ is a pure braid, then there exists a cobordism
given by $2(n-1)$ $1$-handles between $\widehat{\alpha\beta\gamma}$ and $\widehat{\alpha\gamma\beta}$.
\end{enumerate}
\end{lemma}
We remark that in \eqref{item:coba}, we do not claim to control which connected sum of $\widehat{\alpha}$ and $\widehat{\beta}$ is involved. (Recall that for two links $L_1$ and $L_2$ the notion of connected sum $L_1\#L_2$ depends on a choice of component in each link.) We postpone the proof of Lemma~\ref{lem:cobordisms} to after its application in the proofs of Theorem~\ref{thm:sBineq}, where we use \eqref{item:coba}, and Proposition~\ref{prop:Df<=A(n-1)}, where we employ~\eqref{item:cobb}.

For the proof of Theorem~\ref{thm:sBineq}, we observe that there exists a cobordism consisting of $(n-1)nk$ $1$-handles between $\widehat{\beta^{nk}\delta}$
and a $nk$-fold connected sum of $\widehat{\beta}$; we denote %(any choice of)
the latter by $nk\widehat{\beta}$. Indeed, by concatenation of $nk$ cobordism as provided by Lemma~\ref{lem:cobordisms}\eqref{item:coba}, we find such a cobordism %that goes
between $\widehat{\beta^{nk}\delta}$ and a connected sum of $nk$ many
$\widehat{\beta}$ and one $\widehat{\delta}$ (which is an unknot) as desired; compare also \cite[Appendix~A]{feller_hubbard}.
In particular, we have
\begin{equation}\label{eq:chi-chi}
1-\chi_4\left(\widehat{\beta^{nk}\delta}\right)\leq 1-\chi_4\left(nk\widehat{\beta}\right)+nk(n-1)
\leq nk\left(1-\chi_4\left(\widehat{\beta}\right)\right)+nk(n-1),
\end{equation}
where the second inequality follows from $1-\chi_4$ being subadditive under connected sum. % of links.

\begin{proof}[Proof of Theorem~\ref{thm:sBineq}]
%Fix an integer $n\geq 3$ and
Set $t=\frac{2}{n-1}$. For every $\beta\in B_n$, we have %we calculate
\begin{align*}
\fd(\beta)&\overset{\text{\eqref{eq:fdtcviaUpsilon}}}
{=}\frac{\widetilde{\Upsilon}_\beta(t)}{t}+\frac{\wr(\beta)}{2}
%\\&
\overset{\text{\eqref{eq:defhomUpsilon}}}
{=}\lim_{k\to\infty}\frac{\Upsilon_{\widehat{\beta^{nk}\delta}}(t)}{nkt}+\frac{\wr(\beta)}{2}
\\&
\overset{\text{\eqref{eq:Upsilon<=tg4}}}{\leq} \lim_{k\to\infty}\frac{g_4\left(\widehat{\beta^{nk}\delta}\right)}{nk}+\frac{\wr(\beta)}{2}
%\\&
%\overset{%\text{\eqref{eq:2g_4=1-chi}}}
{=}\lim_{k\to\infty}\frac{1-\chi_4\left(\widehat{\beta^{nk}\delta}\right)}{2nk}+\frac{\wr(\beta)}{2}
\\&
\overset{\text{\eqref{eq:chi-chi}}}{\leq}\lim_{k\to\infty}\frac{nk\left(1-\chi_4\left(\widehat{\beta}\right)\right)+nk(n-1)}{2nk}+\frac{\wr(\beta)}{2}
%\\&
\overset{\phantom{\text{\eqref{eq:fdtcviaUpsilon}}}}
{=}\frac{-\chi_4\left(\widehat{\beta}\right)+n}{2}+\frac{\wr(\beta)}{2}
\\&
\overset{\text{\eqref{eq:sBineqwr}}}{\leq}\frac{-\chi_4\left(\widehat{\beta}\right)+n}{2}+\frac{-\chi_4\left(\widehat{\beta}\right)+n}{2}
%\\&
\overset{\phantom{\text{\eqref{eq:fdtcviaUpsilon}}}}
{=}-\chi_4\left(\widehat{\beta}\right)+n.\qedhere
\end{align*}
\end{proof}

%We turn to the proof of Proposition~\ref{prop:Df<=A(n-1)}.
\begin{proof}[Proof of Proposition~\ref{prop:Df<=A(n-1)}]
%We show that for each $\epsilon>0$, $(n-1)A\geq D_f-\epsilon$.
Fix $\epsilon>0$ and let $\alpha$ and $\beta$ be $n$-braids such that $f(\alpha\beta)-f(\alpha)-f(\beta)\geq D_f-\epsilon$.
We first note that we can and do assume that $\alpha$ and $\beta$ are pure braids.
Indeed, if not, pick $\alpha'$ and $\beta'$ such that $f(\alpha'\beta')-f(\alpha')-f(\beta')\geq D_f-\epsilon/n$ and set $\alpha\coloneqq (\alpha')^n$ and $\beta\coloneqq (\beta')^n$. Combining
\[f((\alpha'\beta')^n)-f((\alpha')^n)-f((\beta')^n)=nf(\alpha'\beta')-nf(\alpha')-nf(\beta')\geq nD_f-\epsilon\] with $|f((ab)^n)-f(a^nb^n)|\leq (n-1)D_f$, which one checks by iteratively applying \[\left|f(a^{k-1}b^{k-1})+f(ab)-f(a^kb^k)\right|=\left|f(b^{k-1}a^{k-1})+f(ab)-f(b^{k-1}a^{k-1}(ab))\right|\leq D_f,\]
%(which holds for all homogeneous quasimorphisms $f\colon G\to\R$ with defect $D_f$ and $a,b\in G$),
we have that $f(\alpha\beta)-f(\alpha)-f(\beta)\geq D_f-\epsilon$.

Fix an even positive integer $k$. Using that $f$ is homogeneous and that $f(ab)-f(a)-f(b)\leq D_f$ for all $a,b\in B_n$, we calculate
\begin{align*}
kD_f-k\epsilon&\leq k(f(\alpha\beta)-f(\alpha)-f(\beta))
%\\&
=f\left((\alpha\beta)^{k}\right)+f\left(\alpha^{-k}\right)+f\left(\beta^{-k}\right)\\
&\leq f\left((\alpha\beta)^{k}\alpha^{-k}\right)+D_f+f\left(\beta^{-k}\right)\\
&\leq f\left((\alpha\beta)^{k}\alpha^{-k}\beta^{-k}\right)+D_f+D_f\\
&\leq f\left((\alpha\beta)^{k}\alpha^{-k}\beta^{-k}\delta\right)-f(\delta)+D_f+D_f+D_f\\
&\leq Ag_4(K)+C-f(\delta)+D_f+D_f+D_f,
\end{align*}
where $K$ denotes the closure of $(\alpha\beta)^{k}\alpha^{-k}\beta^{-k}\delta$ and, as above, $\delta=a_1\cdots a_{n-1}$. Note that $K$ is a knot since $(\alpha\beta)^{k}\alpha^{-k}\beta^{-k}$ is a pure braid.

Next we observe that there exists a cobordism of genus $\frac{k}{2}(n-1)$ between $K$ and the closure of $\beta^{k}\alpha^{k}\alpha^{-k}\beta^{-k}\delta=\delta$. % (here we use that $n$ devides $k$).
%To see this, we use the follwing fact. The links $\widehat{\alpha_1\alpha_2\alpha_3}$ and $\widehat{\alpha_1\alpha_2\alpha_3}$ are related by a cobordism given by $2(n-1)$ $1$-handles for all $\alpha_2\alpha_1\alpha_3\in B_n$. This fact follows since there is a connect sum of $\widehat{\alpha_1\alpha_2}=\widehat{\alpha_2\alpha_1}$ and $\widehat{\alpha_3}$ which has a cobordism given by $(n-1)$ $1$-handles to both $\widehat{\alpha_1\alpha_2\alpha_3}$ and $\widehat{\alpha_2\alpha_1\alpha_3}$.

For this, we write $(\alpha\beta)^{k}\alpha^{-k}\beta^{-k}$ as a product of $\frac{k}{2}$ commutators of pure braids, i.e.~$(\alpha\beta)^{k}\alpha^{-k}\beta^{-k}=[\alpha_1,\beta_1][\alpha_2,\beta_2]\cdots[\alpha_\frac{k}{2},\beta_\frac{k}{2}]$ for some pure braids $\alpha_i,\beta_i\in B_n$. This is possible by~\cite[Proof of Lemma~2.24]{Calegari_scl}; compare also~\cite{Bavard_scl}.\footnote{
In the first version of this article, we use a different expression for $(\alpha\beta)^{k}\alpha^{-k}\beta^{-k}$. We are thankful to Tesuya Ito for reminding us that the stable commutator length of a commutator is at most $1/2$, which improved the bound of Proposition~\ref{prop:Df<=A(n-1)} by a factor of $\frac{1}{2}$.}

By Lemma~\ref{lem:cobordisms}\eqref{item:cobb}, for all $b\in B_n$ the closures of $[\alpha_i,\beta_i]b=\alpha_i\beta_i\alpha_i^{-1}\beta_i^{-1}b$ and $\beta_i\alpha_i\alpha_i^{-1}\beta_i^{-1}b=b$ are related by a cobordism with $2(n-1)$ $1$-handles, hence
%the closures of consecutive braids in the following sequence of braids are related by a cobordism given by $2(n-1)$ $1$-handles:
%\begin{align*}(\alpha\beta)^{k}\alpha^{-k}\beta^{-k}\delta&=(\alpha\beta)^{k-1}\alpha\beta\alpha^{-k}\beta^{-k}\delta,\\
%(\alpha\beta)^{k-1}\beta\alpha\alpha^{-k}\beta^{-k}\delta&=(\alpha\beta)^{k-2}\alpha\beta\beta\alpha\alpha{-k}\beta^{-k}\delta,\\
%&\cdots\\
%(\alpha\beta)^{k-i}\beta^i\alpha^{i-1}\alpha^{-k}\beta^{-k}\delta&=(\alpha\beta)^{k-i-1}\alpha\beta\beta^i \alpha^{i}\alpha^{-k}\beta^{-k}\delta,
%\\&\cdots
%\\(\alpha\beta)^{0}\beta^{k}\alpha^{k}\alpha^{-k}\beta^{-k}\delta&=\delta&.\end{align*}
%Composing those $k$ cobordisms gives a cobordism between $K$ and the closure of $\delta$ given by $2k(n-1)$ $1$-handles. In other words, we have a cobordism of genus $k(n-1)$ as desired.
applying this $\frac{k}{2}$ times gives a cobordism between $K$ and the closure of $\delta$ given by $k(n-1)$ $1$-handles. In other words, we have a cobordism of genus $\frac{k}{2}(n-1)$ as desired.

Since $\delta$ has the unknot as its closure, we have $g_4(K)\leq \frac{k}{2}(n-1)$ by the last paragraph.
We conclude that
\[kD_f-k\epsilon\leq A\frac{k}{2}(n-1)+C-f(\delta)+3D_f,\] which yields $D_f\leq \frac{A}{2}(n-1)$ by first dividing by $k$ and taking the limit $k\to\infty$ and then letting $\epsilon$ tend to 0. 
\end{proof}

Finally, we turn to the proof of Lemma~\ref{lem:cobordisms}. The idea of the proof is of a similar flavour as the arguments used in \cite{Brandenbursky_11} and \cite[Appendix~A]{feller_hubbard}, but to the best of our knowledge, the exact statement does not yet appear in the literature.
\begin{proof}[Proof of Lemma~\ref{lem:cobordisms}] To see~\eqref{item:coba}, consider a diagram for $\widehat{\alpha\beta}$ as depicted in Figure~\ref{fig:cobs}~A) (where $\gamma$ is taken to be the trivial braid) and apply $(n-1)$ handle moves, starting with the one indicated by the blackboard framed dotted (green) arc, to find the diagram in Figure~\ref{fig:cobs}~B).
\begin{figure}[h]
\centering
\def\svgscale{0.8}
\input{cobordisms.pdf_tex}
\caption{Isotopies and cobordisms proving Lemma~\ref{lem:cobordisms}. For readability of the diagrams the illustration is for $n=4$.}
\label{fig:cobs}
\end{figure}
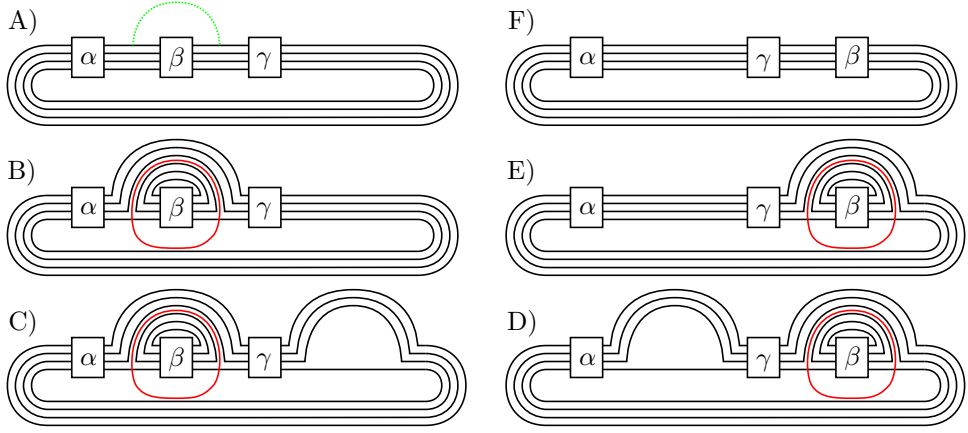
The link given by this diagram is a connected sum of $\widehat{\alpha}$ and $\widehat{\alpha\beta}$ with respect to the indicated sphere (red); hence, we have that there exists a cobordism between
$\widehat{\alpha\beta}$ and the connected sum of $\widehat{\alpha}$ and $\widehat{\alpha\beta}$ depicted in Figure~\ref{fig:cobs}~B).

 We turn to~\eqref{item:cobb}. Since $\alpha\beta\gamma$, $\beta\gamma\alpha$, and $\gamma\alpha\beta$ all are conjugate and hence have the same closure, we may and do assume that $\gamma$ is a pure braid. Consider a diagram for $\widehat{\alpha\beta}$ as depicted in Figure~\ref{fig:cobs}~A) and apply $(n-1)$ handle moves to find the diagram in Figure~\ref{fig:cobs}~B) as in the proof of~\eqref{item:coba}.
 Figure~\ref{fig:cobs}~B) and Figure~\ref{fig:cobs}~C) depict isotopic links (in fact the diagrams are the same up to isotopy of the plane). Figure~\ref{fig:cobs}~C) and Figure~\ref{fig:cobs}~D) depict isotopic links. An isotopy is given by shrinking and moving the summand $\widehat{\beta}$ through $\gamma$ (here we invoke that $\gamma$ is a pure braid). Finally, Figure~\ref{fig:cobs}~D) and Figure~\ref{fig:cobs}~E) depict isotopic links and $(n-1)$ handle moves turn the diagram given in Figure~\ref{fig:cobs}~E) into the one given in Figure~\ref{fig:cobs}~F). All in all we find that there exists a cobordism given by $2(n-1)$ $1$-handles between $\widehat{\alpha\beta\gamma}$ and $\widehat{\alpha\gamma\beta}$ as desired.
\end{proof}
\appendix
\section{The slice-Bennequin inequality for the homogenization of concordance homomorphisms}\label{ap:sBiforI}
In this %short
appendix, we explain %making precise %the statement
that
the homogenization of a concordance homomorphism satisfies a version of the slice-Bennequin inequality. This can be understood as providing a common framework for both the slice-Bennequin inequality~\eqref{eq:sBineqwr} and Theorem~\ref{thm:sBineq}; see Examples~\ref{Ex:I=sti} and~\ref{Ex:I=Upsilon}, respectively.
%This can be seen as a follow up to~\cite[Apendix A]{feller_hubbard}.
What follows below is based on the same idea as the proof of Theorem~\ref{thm:sBineq}, which was rather straight forward once the necessary preparations (like Lemma~\ref{lem:cobordisms}) are made.
Still, we think it is worth making this explicit as the exact statement and perspective appear to be absent from the literature. What follows owes a lot to the ideas of homogenization of concordance homomorphisms as pursued in \cite{GambaudoGhys_BraidsSignatures} for Tristram-Levine signatures, and in general in~\cite{Brandenbursky_11} and the idea of proof of the slice-Bennequin inequality as pioneered by Rudolph~\cite{rudolph_QPasObstruction}. % with~\cite[Appendix~A]{feller_hubbard}.

A real-valued knot invariant $I\colon \mathfrak{K}nots\to \R$ is a \emph{concordance homomorphism}, if $I(K\#J)=I(K)+I(J)$ and $I(K)\leq g_4(K)$ for all $K,J\in \mathfrak{K}nots$, where
$\mathfrak{K}nots$ denotes the set of isotopy classes of knots.\footnote{We note that, as the name `concordance homomorphism' suggests, such $I$ factor through the smooth concordance group and the induced map is a group homomorphism. However, we caution the reader to keep in mind that $I(K)\leq g_4(K)$ is a stronger condition, even if this is not reflected in the name.}  For each concordance invariant $I$,
\begin{equation}\label{eq:defhomI}
\widetilde{I}\to\R,\quad\beta\mapsto\widetilde{I}(\beta)\coloneqq \lim_{k\to\infty}\frac{I\left(\widehat{\beta^{nk}\delta}\right)(t)}{nk},\end{equation}
is a homogeneous quasimorphism with defect $D\leq\frac{n-1}{2}$; see~\cite[Lemma~A.1]{feller_hubbard}.
\begin{lemma}\label{lem:sBforhomI} Fix an integer $1\geq n$. For all $\beta\in B_n$, we have
$\widetilde{I}(\beta)\leq \frac{ -\chi_4\left(\widehat{\beta}\right)+n}{2}$.
%Furthermore, if $\beta$ has a knot as a closure, then we have
%$\widetilde{I}(\beta)\leq  I(\widehat{\beta})-\chi_4\left(\widehat{\beta}\right)+n$. 
%for all $\beta\in B_n$.
\end{lemma}\begin{proof}
$\widetilde{I}(\beta)%&
\overset{\text{\eqref{eq:defhomI}}}
{=}\lim_{k\to\infty}\frac{I\left(\widehat{\beta^{nk}\delta}\right)}{nk}
%\\&
{\leq} \lim_{k\to\infty}\frac{g_4\left(\widehat{\beta^{nk}\delta}\right)}{nk}
{=}\lim_{k\to\infty}\frac{1-\chi_4\left(\widehat{\beta^{nk}\delta}\right)}{2nk}
%\\&
$

\hspace{1.46cm}$\overset{\text{\eqref{eq:chi-chi}}}{\leq}\lim_{k\to\infty}\frac{nk\left(1-\chi_4\left(\widehat{\beta}\right)\right)+nk(n-1)}{2nk}
\overset{\phantom{\text{\eqref{eq:fdtcviaUpsilon}}}}
{=}\frac{-\chi_4\left(\widehat{\beta}\right)+n}{2}.$
\end{proof}

In case $\beta\in B_n$ has a knot as its closure $\widehat{\beta}$, then $\left|I\left(\widehat{\beta}\right)-\widetilde{I}(\beta)\right|\leq \frac{n-1}{2}$ (this follows readily from Lemma~\ref{lem:cobordisms}, it is explicitly state in~\cite[Lemma~A.1]{feller_hubbard}), hence
%one finds the following version of Lemma~\ref{lem:sBforhomI} involving $I\left(\widehat{\beta}\right)$.
\begin{equation}\label{eq:sBforI}
\widetilde{I}(\beta)\leq I\left(\widehat{\beta}\right)+\frac{n-1}{2}\leq g_4\left(\widehat{\beta}\right)+\frac{n-1}{2}=\frac{ -\chi_4\left(\widehat{\beta}\right)+n}{2}.
\end{equation}

%Examples~\ref{Ex:I=sti} and~\ref{Ex:I=Upsilon} below explain how the slice-Bennequin inequality and Theorem~\ref{thm:sBineq}, respectively, can be understood as instances of Lemma~\ref{lem:cobordisms}.
\begin{Example}\label{Ex:I=sti}
%As an example, and 
%Connecting back to the slice-Bennequin inequality~\eqref{eq:sBineqwr},
We consider the case when $I$ is a \emph{slice torus invariant}---a concordance homomorphism $I$ with $I(T_{p,p+1})=g_4(T_{p,p+1})=(p-1)p/2$ for positive integers $p$. Slice torus invariants include Ozsv\'ath-Szab\'o's $\tau$~\cite{OzsvathSzabo_03_KFHandthefourballgenus} and Rasumussen's $s$~\cite{rasmussen_sInv}. In this case we have $\widetilde{I}=\wr/2$; see e.g.~\cite[Lemma~A.3]{feller_hubbard}. Hence, for such $I$, Lemma~\ref{lem:sBforhomI} recovers~\eqref{eq:sBineqwr}, and~\eqref{eq:sBforI} reads, for all $\beta\in B_n$ with closure a knot,
\[\wr(\beta)\leq 2I\left(\widehat{\beta}\right)+n-1\leq 2g_4\left(\widehat{\beta}\right)+{n-1}.\]
This is philosophically pleasing: the slice torus invariants, which are the concordance homomorphisms that are strong enough to reprove the local Thom conjecture (i.e.~$g_4(T_{p,p+1})=(p-1)p/2$ for all positive integers $p$~\cite{KronheimerMrowka_GenusofEmb}), homogenize to $\wr/2$ and hence recover the slice-Bennequin inequality, which Rudolph derived using only the local Thom conjecture as an elementary input.    
\end{Example}

\begin{Example}\label{Ex:I=Upsilon} If %instead
%we consider
$I(K)\coloneqq \frac{\Upsilon_K(\frac{2}{n-1})}{n-1}+\frac{\tau(K)}{2}$,
%we have
then $\widetilde{I}=2\fd$ by~\eqref{eq:fdtcviaUpsilon}. Hence, Lemma~\ref{lem:sBforhomI} yields Theorem~\ref{thm:sBineq}, and~\eqref{eq:sBforI} reads, for all $\beta\in B_n$ with closure a knot,
\[\fd(\beta)\leq \frac{2\Upsilon_{\widehat{\beta}}(\frac{2}{n-1})}{n-1}+\tau\left(\widehat{\beta}\right)+n-1\leq 2g_4\left(\widehat{\beta}\right)+{n-1}.\]
\end{Example}
In light of \eqref{eq:sBforI}, we wonder whether every homogeneous quasimorphism that satisfies a slice-Bennequin inequality does arise as a homogenization. 
\begin{question}\label{q:hqmswithSBIareHom}
Fix $n\geq 3$ and 
let $f\colon B_n\to\R$ be a homogeneous quasimorphism. If there exist constants $A,C\in\R$ such that
\[\left|f(\beta)\right|\leq Ag_4\left(\widehat{\beta}\right)+C\text{ for all $\beta\in B_n$ with closure a knot},\]
Does there exist a concordance homomorphism $I$ and $r\in\R$ such that $f=r\widetilde{I}$?
\end{question}
In light of the fact that Question~\ref{q:qmwithoutSBI}, as far as we know, remains open, it seems that even a positive answer to the following is possible.
\begin{question}\label{q:hqmsareHom}
Fix $n\geq 3$ and 
let $f\colon B_n\to\R$ be a homogeneous quasimorphism.
Does there exist a concordance homomorphism $I$ and $r\in\R$ such that $f=r\widetilde{I}$?
\end{question}
This author strongly suspects that the answer to Question~\ref{q:hqmsareHom} is no, but is unable to provide a counterexample.

 \bibliographystyle{alpha} \bibliography{peterbib}
\end{document}

%% file: Cobordisms.pdf_tex
%% Creator: Inkscape 1.0.2-2 (e86c870879, 2021-01-15), www.inkscape.org
%% Accompanies image file 'Cobordisms.pdf' (pdf, eps, ps)
\begingroup%
  \makeatletter%
  \providecommand\color[2][]{%
    \errmessage{(Inkscape) Color is used for the text in Inkscape, but the package 'color.sty' is not loaded}%
    \renewcommand\color[2][]{}%
  }%
  \providecommand\transparent[1]{%
    \errmessage{(Inkscape) Transparency is used (non-zero) for the text in Inkscape, but the package 'transparent.sty' is not loaded}%
    \renewcommand\transparent[1]{}%
  }%
  \providecommand\rotatebox[2]{#2}%
  \newcommand*\fsize{\dimexpr\f@size pt\relax}%
  \newcommand*\lineheight[1]{\fontsize{\fsize}{#1\fsize}\selectfont}%
  \ifx\svgwidth\undefined%
    \setlength{\unitlength}{446.99986736bp}%
    \ifx\svgscale\undefined%
      \relax%
    \else%
      \setlength{\unitlength}{\unitlength * \real{\svgscale}}%
    \fi%
  \else%
    \setlength{\unitlength}{\svgwidth}%
  \fi%
  \global\let\svgwidth\undefined%
  \global\let\svgscale\undefined%
  \makeatother%
  \begin{picture}(1,0.4463105)%
    \lineheight{1}%
    \setlength\tabcolsep{0pt}%
    \put(0,0){\includegraphics[width=\unitlength,page=1]{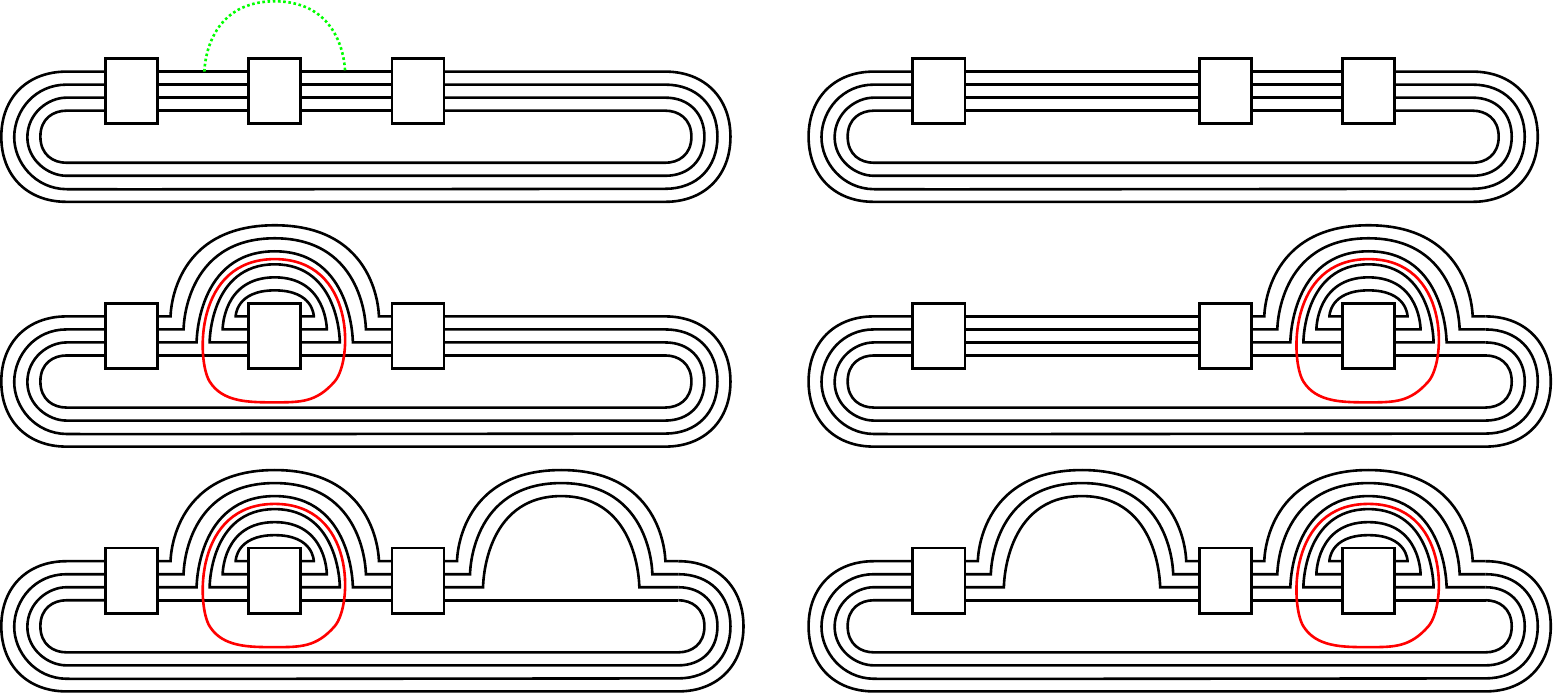}}%
    \put(0.07657569,0.38038955){\color[rgb]{0,0,0}\makebox(0,0)[lt]{\lineheight{1.25}\smash{\begin{tabular}[t]{l}$\alpha$\end{tabular}}}}%
    \put(0.16883817,0.37883194){\color[rgb]{0,0,0}\makebox(0,0)[lt]{\lineheight{1.25}\smash{\begin{tabular}[t]{l}$\beta$\end{tabular}}}}%
    \put(0.26073273,0.38038955){\color[rgb]{0,0,0}\makebox(0,0)[lt]{\lineheight{1.25}\smash{\begin{tabular}[t]{l}$\gamma$\end{tabular}}}}%
    \put(0.07657569,0.22267143){\color[rgb]{0,0,0}\makebox(0,0)[lt]{\lineheight{1.25}\smash{\begin{tabular}[t]{l}$\alpha$\end{tabular}}}}%
    \put(0.16883817,0.22111387){\color[rgb]{0,0,0}\makebox(0,0)[lt]{\lineheight{1.25}\smash{\begin{tabular}[t]{l}$\beta$\end{tabular}}}}%
    \put(0.26073273,0.22267143){\color[rgb]{0,0,0}\makebox(0,0)[lt]{\lineheight{1.25}\smash{\begin{tabular}[t]{l}$\gamma$\end{tabular}}}}%
    \put(0.07657569,0.06435381){\color[rgb]{0,0,0}\makebox(0,0)[lt]{\lineheight{1.25}\smash{\begin{tabular}[t]{l}$\alpha$\end{tabular}}}}%
    \put(0.16883817,0.06279623){\color[rgb]{0,0,0}\makebox(0,0)[lt]{\lineheight{1.25}\smash{\begin{tabular}[t]{l}$\beta$\end{tabular}}}}%
    \put(0.26073273,0.06435381){\color[rgb]{0,0,0}\makebox(0,0)[lt]{\lineheight{1.25}\smash{\begin{tabular}[t]{l}$\gamma$\end{tabular}}}}%
    \put(0.59700955,0.38038955){\color[rgb]{0,0,0}\makebox(0,0)[lt]{\lineheight{1.25}\smash{\begin{tabular}[t]{l}$\alpha$\end{tabular}}}}%
    \put(0.59700955,0.22267143){\color[rgb]{0,0,0}\makebox(0,0)[lt]{\lineheight{1.25}\smash{\begin{tabular}[t]{l}$\alpha$\end{tabular}}}}%
    \put(0.59700955,0.06435381){\color[rgb]{0,0,0}\makebox(0,0)[lt]{\lineheight{1.25}\smash{\begin{tabular}[t]{l}$\alpha$\end{tabular}}}}%
    \put(0.78195625,0.37883194){\color[rgb]{0,0,0}\makebox(0,0)[lt]{\lineheight{1.25}\smash{\begin{tabular}[t]{l}$\gamma$\end{tabular}}}}%
    \put(0.78195626,0.22111382){\color[rgb]{0,0,0}\makebox(0,0)[lt]{\lineheight{1.25}\smash{\begin{tabular}[t]{l}$\gamma$\end{tabular}}}}%
    \put(0.78195626,0.0627962){\color[rgb]{0,0,0}\makebox(0,0)[lt]{\lineheight{1.25}\smash{\begin{tabular}[t]{l}$\gamma$\end{tabular}}}}%
    \put(0.8738031,0.38038955){\color[rgb]{0,0,0}\makebox(0,0)[lt]{\lineheight{1.25}\smash{\begin{tabular}[t]{l}$\beta$\end{tabular}}}}%
    \put(0.87380311,0.22267143){\color[rgb]{0,0,0}\makebox(0,0)[lt]{\lineheight{1.25}\smash{\begin{tabular}[t]{l}$\beta$\end{tabular}}}}%
    \put(0.87380311,0.06435381){\color[rgb]{0,0,0}\makebox(0,0)[lt]{\lineheight{1.25}\smash{\begin{tabular}[t]{l}$\beta$\end{tabular}}}}%
    \put(0.00168826,0.41744661){\color[rgb]{0,0,0}\makebox(0,0)[lt]{\lineheight{1.25}\smash{\begin{tabular}[t]{l}A)\end{tabular}}}}%
    \put(0.00084057,0.26015224){\color[rgb]{0,0,0}\makebox(0,0)[lt]{\lineheight{1.25}\smash{\begin{tabular}[t]{l}B)\end{tabular}}}}%
    \put(0.00182282,0.1015867){\color[rgb]{0,0,0}\makebox(0,0)[lt]{\lineheight{1.25}\smash{\begin{tabular}[t]{l}C)\end{tabular}}}}%
    \put(0.52212008,0.1015867){\color[rgb]{0,0,0}\makebox(0,0)[lt]{\lineheight{1.25}\smash{\begin{tabular}[t]{l}D)\end{tabular}}}}%
    \put(0.52212008,0.26015224){\color[rgb]{0,0,0}\makebox(0,0)[lt]{\lineheight{1.25}\smash{\begin{tabular}[t]{l}E)\end{tabular}}}}%
    \put(0.52300813,0.41744661){\color[rgb]{0,0,0}\makebox(0,0)[lt]{\lineheight{1.25}\smash{\begin{tabular}[t]{l}F)\end{tabular}}}}%
  \end{picture}%
\endgroup%